\documentclass[12pt]{amsart}
\usepackage{graphicx}

\newtheorem{thm}{Theorem}[section]
\newtheorem{cor}[thm]{Corollary}
\newtheorem{lem}[thm]{Lemma}
\newtheorem{prop}[thm]{Proposition}
\theoremstyle{definition}
\newtheorem{defn}[thm]{Definition}
\theoremstyle{remark}

\numberwithin{equation}{section}

\usepackage{graphicx}
\usepackage{amsmath}

\begin{document}
\title[G-frames]
 {G-frame representation and Invertibility of g-Bessel Multipliers}
\author{A. Abdollahi}%
\author{E. Rahimi}

\address{Department of Mathematics, College of Sciences, Shiraz University, Shiraz 71454, Iran}

\email{abdolahi@shirazu.ac.ir}%
\email{rahimie@shirazu.ac.ir}

\thanks{The first author was supported by a grant from the Shiraz university Research Council.}

\subjclass[2000]{Primary 42C15; Secondary 41A58}
 \keywords{g-frames,g-orthonormal basis, controlled
g-frames, weighted g-frames, g-frame multipliers}

\keywords{g-frames, g-orthonormal bases, g-Riesz bases, g-frame
multipliers, controlled g-frames,
weighted g-frames}%

\maketitle
\begin{abstract}
In this paper we show that every g-frame for an \linebreak
infinite dimensional Hilbert space $\mathcal{H}$ can be written as
a sum of three g-orthonormal bases for $\mathcal{H}$. Also, we
prove that every g-frame can be represented as a linear
combination of two g-orthonormal bases if and only if it is a
g-Riesz basis. Further, we show each g-Bessel multiplier is a
Bessel multiplier and investigate the inversion of g-frame
multipliers. Finally, we introduce the concept of controlled
g-frames and weighted g-frames and show that the sequence induced
by each controlled g-frame (resp. weighted g-frame) is a
controlled frame (resp. weighted frame).
\end{abstract}
\section{Introduction}
Frames for a separable  Hilbert space  were first introduced in
1952 by Duffin and Schaeffer \cite{R. J. Duffin.52}. In \cite{S.
Wenchang.06}, a generalization of the frame concept was
introduced. Sun introduced g-frames and  g-Riesz bases in a
complex Hilbert space and discussed some properties of them.
G-frames and g-Riesz bases in complex Hilbert spaces have some
properties similar to those of frames, Riesz bases, but not all
the properties are similar (see
\cite{S. Wenchang.06}). In this paper we generalize some results in  \cite{p.Balazs.09, p.Balazs.06, P.G.Casazza.98}, from frame theory to g-frames.\\
Throughout this paper, $\mathcal{H}$ and $\mathcal{K}$ are
separable Hilbert spaces and $\{\mathcal{H}_{i}\}_{i \in J}
\subseteq \mathcal{K}$ is a sequence of separable Hilbert spaces,
where $J$ is a subset of $\mathbb{Z}$,
$\mathcal{L}(\mathcal{H},\mathcal{H}_{i})$ is the collection of
all bounded linear operators from $\mathcal{H}$ to
$\mathcal{H}_{i}$.
For each sequence $\{\mathcal{H}_{i}\}_{i \in J}$, we define the
space $(\bigoplus_{i \in J} \mathcal{H}_{i})_{l_{2}}$ by

$$( \bigoplus_{i \in J} \mathcal{H}_{i})_{l_{2}}  =  \{ \{f_{i}\} _{i
\in J}:f_{i} \in \mathcal{H}_{i} ,~ i \in J~ and ~\sum_{i \in
J}\|f_{i}\|^{2} < \infty \}.$$ With the inner product defined by
$$\langle\{f_{i}\},\{g_{i}\}\rangle=\sum_{i \in J}\langle
f_{i},g_{i}\rangle,$$ it is clear that $(\bigoplus_{i \in J}
\mathcal{H}_{i})_{l_{2}}$ is a Hilbert space.\\

A {\it frame} for a complex Hilbert space $\mathcal{H}$ is a family of
vectors $\{f_{i}\}_{i \in J}$ so that there are two positive
constants A and B satisfying
$$A\|f\|^{2}\leq\sum_{i\in J}|\langle{f,f_{i}}\rangle|^{2}\leq B\|
f\|^{2}, f\in \mathcal{H}.$$ The constants A and B are called
{\it lower} and {\it upper frame bounds}.
The frame possesses many nice properties which makes it very useful in wavelet analysis, irregular sampling theory, signal processing and many other fields. We infer to \cite{R.Balan.P.G.Casazza.C.Heil.Z.Landau.2003, R.Balan.P.G.Casazza.C.Heil.Z.Landau.2006, O.Christensen.03, I.Daubecheies.A.Grossmann.Y.Meyer.1986, I.Daubechies.1992, K.Gröchenig.2001, R.Young.2001}. \\

\noindent The notion of frames has been generalized to g-frames by W. Sun (\cite{S. Wenchang.06}) in the following way:\\
A sequence $\{\Lambda_{i}\}_{i \in J}$ is called a {\it generalized frame}, or simply a {\it g-frame},
 for $\mathcal{H}$ with respect to $\{\mathcal{H}_{i}\}_{i \in J}$ if there exist two positive
constants A and B such that, for all $f\in \mathcal{H}$,
$$A\|f\|^{2}\leq\sum_{i\in J}\|\Lambda_{i}f\|^{2}\leq B\| f\|^{2}.$$
 The constants A and B are called the {\it lower} and {\it upper g-frame bounds},
respectively. The supremum of all such $A$ and the infimum of all
such $B$ are  called the optimal bounds. If $A=B$ we call this
g-frame a {\it tight g-frame} and if $A = B=1$ it is called a {\it normalized
tight g-frame}.  We say simply a g-frame for $\mathcal{H}$, and
denote by $\{\Lambda_{i}\}_{i \in J}$, whenever the space sequence
$\mathcal{H}_{i}$ and the index set $J$ are clear. If we only have
the upper bound, we call $\{\Lambda_{i}\}_{i \in J}$ a {\it g-Bessel
sequence} with bound B. We say that $\{\Lambda_{i}\}_{i \in J}$ is
{\it g-complete}, if $\bigcap_{i \in J} \{ f :\Lambda_{i} f=0\}=\{0\}$
and is called {\it g-orthonormal basis} for $\mathcal{H}$, if
$$\langle\Lambda_{i}^{\star}g_{i},\Lambda_{j}^{\star}g_{j}\rangle = \delta_{i,j}\langle g_{i},g_{j}\rangle, ~i,j \in J,~ g_{i}
\in \mathcal{H}_{i} ,~g_{j} \in \mathcal{H}_{j},$$ and
$$\sum_{i\in J}\|\Lambda_{i}f\|^{2}=\|f\|^{2}, f \in \mathcal{H}$$

We say that $\{\Lambda_{i}\}_{i \in J}$ is a {\it g-Riesz basis} for
$\mathcal{H}$, if it is g-complete and there exist constants $0 <A
\leq B < \infty$, such that for any finite subset $I\subseteq J$
and $g_{i}\in \mathcal{H}_{i}$, $ i \in I$,
$$A\sum_{i \in I}\|g_{i}\|^{2}\leq\|\sum_{i\in
I}\Lambda_{i}^{\star}g_{i}\|^{2}\leq B\sum_{i \in
I}\|g_{i}\|^{2}.$$
In \cite{S. Wenchang.06}, for the g-frame $\{ \Lambda_{i}\}_{i \in
J}$, the {\it g-frame operator} $S$ is defined by
$$S : \mathcal{H} \rightarrow \mathcal{H}, Sf = \sum_{i \in J} \Lambda^{\star}_{i}
\Lambda_{i}f, $$ which is a bounded, self-adjoint, positive and
invertible operator and $$A \leq \|S\| \leq B.$$ The {\it canonical
dual g-frame} for $\{ \Lambda_{i}\}_{i \in J}$ is defined by $\{\widetilde{\Lambda_{i}}\}_{i \in J}$, where
$\widetilde{\Lambda_{i}}=\Lambda_{i}S^{-1}$, which is also a g-frame for $\mathcal{H}$ with $\frac{1}{B}$ and
$\frac{1}{A}$ as its lower and upper g-frame bounds, respectively.
Also every $f \in \mathcal{H}$ has an expansion
$$f = \sum_{i \in J} S^{-1} \Lambda^{\star}_{i} \Lambda_{i}f =
\sum_{i \in J} \Lambda^{\star}_{i} \Lambda_{i}S^{-1}f.$$\\

Let $\{ \Lambda_{i}\}_{i \in J}$ be a sequence in
$\mathcal{L}(\mathcal{H},\mathcal{H}_{i})$, $\{e_{i,k}:k \in
K_{i}\}$ be an orthonormal basis for $\mathcal{H}_{i}$, $i \in J$
where $K_{i}$ is a subset of $\mathbb{Z}$ and let
$\psi_{i,k}=\Lambda_{i}^{\star}e_{i,k}$. We have
$\Lambda_{i}f=\sum_{k \in K_{i}}\langle f,\psi_{i,k}\rangle
e_{i,k}$. We call $\{\psi_{i,k}: i \in J, k \in K_{i} \}$ the
sequence induced by $\{\Lambda_{i}\}_{i \in J}$ with respect to
$\{e_{i,k}:k \in K_{i}\}$.\\

In order to present the main results of this paper, we need the
following Theorem that describe the relationship between frame
(resp. Bessel sequence, tight frame, Riesz basis, orthonormal
basis) and g-frame (resp. g-Bessel sequence, tight g-frame,
g-Riesz basis, g-orthonormal basis), which can be found in
\cite{S. Wenchang.06}.

\begin{thm}\label{thm5}  Let $ \Lambda_{i} \in \mathcal{L}(\mathcal{H},\mathcal{H}_{i})$, and $\psi_{i,k}$ be defined
as above. Then we have the followings:
\begin{itemize}
\item[i)] $\{\Lambda_{i}\}_{i \in J}$ is a g-frame (resp. g-Bessel
sequence, tight g-frame, g-Riesz basis, g-orthonormal basis) for
$\mathcal{H}$ if and only if $\{\psi_{i,k}: i \in J, k \in K_{i}
\}$ is a frame (resp. Bessel sequence, tight frame, Riesz basis,
orthonormal basis ) for $\mathcal{H}$. \item[ii)] The g-frame
operator for  $\{\Lambda_{i}\}_{i \in J}$
 coincides
with the frame operator for $\{\psi_{i,k}: i \in J, k \in K_{i}
\}$. \item[iii)] Moreover, $\{\Lambda_{i}\}_{i \in J}$ and
$\{\widetilde{\Lambda}_{i}\}_{i \in J}$ are a pair of (canonical)
dual g-frames if and only if the induced sequences are a pair of
(canonical) dual frames.
\end{itemize}
\end{thm}

We call $\{\psi_{i,k}:   i \in J, k \in K_{i} \}$ the sequence
induced by  $\{\Lambda_{i}\}_{i \in J}$     with respect to
$\{e_{i,k}:k \in K_{i}\}$.
 If $\{e'_{i,k}:   i \in J, k \in K_{i} \}$ is an
orthonormal basis for $\mathcal{H}$ and $\Theta_{i}f=\sum_{k \in
K_{i}}\langle f,e'_{i,k}\rangle e_{i,k}$, then $\{\Theta_{i}\}_{i
\in J}$ is a g-orthonormal basis for $\mathcal{H}$. \\

Given two Bessel sequences, $\Psi = (\psi_{i})$ and $\Phi =
(\phi_{i})$, and the weight sequence $m = (m_{i})$, the {\it Bessel
multiplier} for these sequences is an operator defined by
$$\mathbb{M}_{m,\Psi,\Phi}f=\sum_{i \in J}  m_{i} \langle
f,\psi_{i}\rangle \phi_{i}.$$ We shorten the notation by setting
$\mathbb{M}_{m,\Psi}=\mathbb{M}_{m,\Psi,\Psi}$ (see
\cite{p.Balazs.07}). Bessel multipliers and, in particular, frame
multipliers have useful applications.
For example, in \cite{P.Balazs.2008}, frame multipliers are used to solve approximation problems. \\

\noindent The concept of Bessel multipliers can be generalized to g-Bessel
as follows: For given g-Bessel sequences
$\Lambda=\{\Lambda_{i}\}_{i \in J}$ and $\Theta=\{\Theta_{i}\}_{i
\in J}$, and the weight sequence $m = (m_{i})$, the {\it g-Bessel
multiplier} is defined by
$$\mathbb{M}_{m,\Lambda,\Theta}f=\sum_{i \in J}  m_{i}
\Lambda_{i}^{\star}\Theta_{i}f.$$
 We shorten the
notation by setting
$\mathbb{M}_{m,\Lambda}=\mathbb{M}_{m,\Lambda,\Lambda}$. (see\cite{A.Rahimi.09})\\

\noindent A sequence $(m_{i} : i \in J)$ is called {\it semi-normalized} if there
are bounds $b\geq a>0$, such that $a \leq |m_{i}| \leq b$ for all
$i \in J$.

We define $ GL(\mathcal{H})$ as the set of all bounded linear
operators with a bounded inverse.
A frame controlled by an operator $C \in GL(\mathcal{H})$ is a family of vectors $\Psi=(\psi_{i} \in \mathcal{H} : i \in J)$, such that there exist two constants   $m_{CL} > 0$ and $ M_{CL} < \infty$ satisfying\\
$$ m_{CL} \|f\|^{2} \leq \sum _{i \in J}\langle f, \psi_{i} \rangle
\langle C\psi_{i},f \rangle \leq M_{CL}\|f\|^{2}, f \in
\mathcal{H}.$$ We call
$$S_{C}f=\sum_{i \in J} \langle f,\psi_{i} \rangle C\psi_{i},$$
the {\it controlled frame operator}.(see \cite{p.Balazs.06})\\

\noindent We also generalize this concept to g-frames. A {\it g-frame controlled} by
the operator $C \in GL(\mathcal{H})$ or {\it $C$-controlled g-frame} is
a family of operators $\Lambda=\{\Lambda_{i}\}_{i \in J}$, such
that there exist two
constants   $  m_{CL} > 0$ and $ M_{CL} < \infty$ satisfying\\

$$ m_{CL} \|f\|^{2} \leq \sum _{i \in J}\langle
\Lambda_{i}C^{\star}f,\Lambda_{i}f \rangle \leq M_{CL}\|f\|^{2},$$

The paper is organized as follows. In  Section 2 we show that
every g-frame for an infinite dimensional Hilbert space
$\mathcal{H}$ can be written as a sum of three g-orthonormal bases
for $\mathcal{H}$. We next show that a g-frame can be represented
as a linear combination of two g-orthonormal bases if and only if
it is a g-Riesz basis. We further show that every g-frame can be
written as a sum of two tight g-frames with g-frame bounds one or
a sum of a g-orthonormal basis and a g-Riesz basis for
$\mathcal{H}$. In Section 3 we show each g-Bessel multiplier is a
Bessel multiplier and investigate the inversion of g-frame
multipliers. Also sufficient conditions for invertibility of
multipliers are determined. In Section 4 we introduce controlled
g-frames and show that the sequence induced by each controlled
g-frame is a controlled frame and controlled g-frames are
equivalent to standard g-frames. Finally, in the last section, we
investigate the concept of weighted g-frames, and show that the
sequence induced by each weighted g-frame is a weighted frame.


\section{Some g-frame Representations}
In \cite{P.G.Casazza.98}, the author has shown, using operator
theory, that every frame in a Hilbert space $\mathcal{H}$ can be
written as the sum of three orthonormal bases. More precisely, if
$(x_i)_{i\in J}$ is a frame for $\mathcal{H}$, then there exist
orthonormal bases $(f_i)$, $(g_i)$ and $(h_i)$ such that
$x_i=a(f_i+g_i+h_i), i\in J$, for some constant $a$. Furthermore,
the author provided an example of a tight frame $(x_i)_{i\in J}$
that cannot be written in the form $x_i=a\, f_i +b\, g_i, i\in J$,
for any orthonormal sequences $(f_i)$, $(g_i)$ and any choice of
constants $a$ and $b$. The author also proved related results, in
particular the following one: a frame in $\mathcal{H}$ can be
written as a linear combination of two orthonormal
 bases if and only if it is a Riesz basis. In this section we generalize some of these results from the frame case to
  the g-frame case.

\begin{prop}\label{thm6}
If $\{ \Lambda_{i} \}_{i \in J}$ is a g-frame for a Hilbert space
$\mathcal{H}$, there are g-orthonormal bases $\{ \Upsilon_{i}\}$,
$\{ \Gamma_{i}\}$, $\{ \Psi_{i}\}$ for $\mathcal{H}$ and a
constant $a$ so that
$\Lambda_{i}=a(\Upsilon_{i}+\Gamma_{i}+\Psi_{i})$ for all $i \in
J$.
\end{prop}
\begin{proof} Let $\{\psi_{i,k}\}$ the sequence induced by $\{\Lambda_{i}\}_{i \in J}$
with respect to an orthonormal basis $\{e_{i,k}:k \in K_{i}\}$.\\
Hence $\Lambda_{i}f=\sum_{k \in K_{i}}\langle f,\psi_{i,k}\rangle e_{i,k}$, and so,
by Corollary 2.2 of \cite{P.G.Casazza.98}, there are constant $a$
and orthonormal bases $\{f_{i,k}\},\{g_{i,k}\},\{h_{i,k}\}$ such
that $\psi_{i,k}=a(f_{i,k}+g_{i,k}+h_{i,k})$. Since
$\{f_{i,k}\},\{g_{i,k}\},\{h_{i,k}\}$ are orthonormal bases for
$\mathcal{H}$, by Theorem \ref{thm5}, $\Upsilon_{i}$,
$\Gamma_{i}$, $\Psi_{i}$ are g-orthonormal bases, where
$\Upsilon_{i}f=\sum_{k \in K_{i}}\langle f,f_{i,k}\rangle
e_{i,k}$, $\Gamma_{i}f=\sum_{k \in K_{i}}\langle f,g_{i,k}\rangle
e_{i,k}$, and $\Psi_{i}f=\sum_{k \in K_{i}}\langle
f,h_{i,k}\rangle e_{i,k}$. The proof is complete by noting that
$\Lambda_{i}=a(\Upsilon_{i}+\Gamma_{i}+\Psi_{i})$ for all $i \in
J$.
\end{proof}

\begin{prop}\label{thm8}A g-frame
$\{ \Lambda_{i}\}_{i \in J }$  can be written as a linear
combination of two g-orthonormal bases for $\mathcal{H}$ if and
only if $\{ \Lambda_{i}\}_{i \in J}$ is a g-Riesz basis for
$\mathcal{H}$.
\end{prop}

\begin{proof}
Let $\{\psi_{i,k}\}$ be the sequence induced by $\{\Lambda_{i}\}_{i \in J}$
with respect to an orthonormal basis $\{e_{i,k}:k \in K_{i}\}$.\\
Then $\Lambda_{i}f=\sum_{k \in K_{i}}\langle f,\psi_{i,k}\rangle e_{i,k}$.
Suppose that there are g-orthonormal
bases $\{ \Upsilon_{i}\}$, $\{ \Gamma_{i}\}$ for $\mathcal{H}$ and
constants $a,b$ such that $\Lambda_{i}=a\Upsilon_{i}+b\Gamma_{i}$
for all $i \in J$. So, by Theorem \ref{thm5}, there are
orthonormal bases $\{f_{i,k}\},\{g_{i,k}\}$ for $\mathcal{H}$ such
that $\Upsilon_{i}f=\sum_{k \in K_{i}}\langle f,f_{i,k}\rangle
e_{i,k}$, $\Gamma_{i}f=\sum_{k \in K_{i}}\langle f,g_{i,k}\rangle
e_{i,k}$. Therefore  $\psi_{i,k}=af_{i,k}+bg_{i,k}$, and
Proposition 2.5 of \cite{P.G.Casazza.98} implies that
$\{\psi_{i,k}\}$ is a Riesz basis. So $\{ \Lambda_{i} \}_{i \in
J}$ is a g-Riesz basis for $\mathcal{H}$ by Theorem \ref{thm5}.
Conversely, if $\{\Lambda_{i}\}_{i \in J}$ is a g-Riesz basis, we
have $\Lambda_{i}f=\sum_{k \in K_{i}}\langle f,\psi_{i,k}\rangle
e_{i,k}$, where $\{\psi_{i,k}\}$ is a Riesz basis. So by
Proposition 2.5 of \cite{P.G.Casazza.98}, for some constants
$a,b$, and orthonormal bases $\{f_{i,k}\}$ and $\{g_{i,k}\}, \psi_{i,k}=af_{i,k}+bg_{i,k}$. Hence $\Lambda_{i}=a\Upsilon_{i}+b\Gamma_{i}$,
 where $\Upsilon_{i}$ and $\Gamma_{i}$ are g-orthonormal bases and $\Upsilon_{i}f=\sum_{k \in K_{i}}\langle
f,f_{i,k}\rangle e_{i,k}$, $\Gamma_{i}f=\sum_{k \in K_{i}}\langle
f,g_{i,k}\rangle e_{i,k}$.
\end{proof}

\begin{prop}\label{thm9} If $K$ is a co-isometry on $\mathcal{H}$, and if
$\{\Theta_{i}\}_{i \in J}$ is a g-orthonormal basis for
$\mathcal{H}$, then $\{\Theta_{i}K^{\star}: i \in J\}$ is a
normalized tight g-frame for $\mathcal{H}$.
\end{prop}

\begin{proof}Since $K$ is a co-isometry, $K^{\star}$ is an isometry. Hence, for all $f \in \mathcal{H}$,
$$\sum_{i\in
J}\|\Theta_{i}K^{\star}f\|^{2}=\|K^{\star}f\|^{2}=\|f\|^{2}.$$
\end{proof}

By the same argument as above, we obtain the following results.
\begin{prop}\label{thm10} Every g-frame is
the sum of two normalized tight g-frames for $\mathcal{H}$.
\end{prop}
\begin{prop}\label{thm11} Every g-frame
for a Hilbert space $\mathcal{H}$ is the sum of a g-orthonormal
basis for $\mathcal{H}$ and a g-Riesz basis for $\mathcal{H}$.
\end{prop}

\section{Invertibility of Multipliers}
In this section we show  each g-Bessel multiplier is a Bessel
multiplier and investigate the inversion of g-frame multipliers.
Also sufficient conditions for invertibility of multipliers are
determined.
Equivalent results as proved in \cite{p.Balazs.07} for Bessel
multiplier can be shown for g-Bessel multiplier. We prove some of
them, the proof of the others follows in the same manner.

The following proposition gives the connection between the
g-Bessel sequences and Bessel sequences.

\begin{prop}\label{thm14} Each g-Bessel multiplier is a Bessel
multiplier. \\
Furthermore, if $m \in \ell^{\infty}$ and
$\Lambda=\{\Lambda_{i}\}_{i \in J}$, $\Theta=\{\Theta_{i} \}_{i
\in J}$  are g-Bessel sequences for $\mathcal{H}$ with bounds
$B_{\Lambda}, B_{\Theta}$, respectively, then the multiplier
$M_{m, \Lambda, \Theta}$ is well defined on $\mathcal{H}$ and
$\|M_{m, \Lambda, \Theta}\| \leq \sqrt{B_{\Lambda} B_{\Theta}}
\|m\|_\infty$.
\end{prop}
\begin{proof} Let $\Lambda=\{\Lambda_{i}\}_{i \in J}$ and
$\Theta=\{\Theta_{i} \}_{i \in J}$ be g-Bessel sequences with induced sequences $\{\psi_{i,k}\}$ and $\{\phi_{i,k}\}$, respectively. Then
$$\mathbb{M}_{m,\Lambda,\Theta}f=\sum_{i \in J}  m_{i}
\Lambda_{i}^{\star}\Theta_{i}f=\sum_{i \in J}\sum_{k \in
K_{i}}m_{i}\langle f,\phi_{i,k}\rangle
\psi_{i,k}=\mathbb{M}_{m',\Phi,\Psi},$$ where $\Psi=\{\psi_{i,k}:
i\in J,k \in K_{i}\}, \Phi=\{\phi_{i,k}: i\in J,k \in K_{i}\}$ and
$m'=\{m'_{i,k}=m_{i}: i\in J,k \in K_{i}\}$. \\ For the proof of
the second part, since the bounds of $\{\psi_{i,k}\}$ and
$\{\varphi_{i,k}\}$ are $B_\Lambda$ and $B_\Theta$, respectively,
the assertion follows by the first part and Theorem 6.1 of
\cite{p.Balazs.07}.
\end{proof}
\begin{thm}\label{thm18}
Let $M_{m, \Lambda, \Theta}$ be well defined and invertible on
$\mathcal{H}$.
\begin{itemize}
\item[i)] If $\Theta=\{\Theta_{i} \}_{i \in J}$ (resp.
 $\Lambda=\{\Lambda_{i}\}_{i \in J}$ is a
g-Bessel sequence for $\mathcal{H}$ with bound $B_{\Theta}$, then
 $m\Lambda=\{m_{i}\Lambda_{i} \}_{i \in J}$ (resp.
$m\Theta$) satisfies the lower g-frame condition for $\mathcal{H}$
with bound $\frac{1}{B_{\Theta} \|M^{-1}_{m, \Lambda,
\Theta}\|^{2}}$. \item[ii)] If  $\Theta$ (resp. $\Lambda$) and
$m\Lambda$ (resp. $m\Theta$) are g-Bessel sequences for
$\mathcal{H}$, then they are g-frames for $\mathcal{H}$.
\item[iii)] If $\Theta$ (resp. $\Lambda$) is a g-Bessel sequence
for $\mathcal{H}$ and $m \in \ell^{\infty}$, then $\Lambda$ (resp.
$\Theta$) satisfies the lower g-frame condition for $\mathcal{H}$.
\item[iv)] If $\Theta$ and $\Lambda$ are g-Bessel sequences for
$\mathcal{H}$ and $m \in \ell^{\infty}$, then $\Theta$ and
$\Lambda$ are g-frames for $\mathcal{H}$; $m\Lambda$ and $m\Theta$
are also g-frames for $\mathcal{H}$.
\end{itemize}
\end{thm}
\begin{proof} (i) Since $\Theta=\{\Theta_{i} \}_{i \in J}$ is a g-Bessel sequence, by Theorem \ref{thm5}
and \ref{thm14}, $\Theta_{i}f=\sum_{k \in K_{i}}\langle
f,\phi_{i,k}\rangle e_{i,k}$, where $\{\phi_{i,k}\}$ is a Bessel
sequence for $\mathcal{H}$ with bound $B_{\Theta}$, and
$\mathbb{M}_{m,\Lambda,\Theta}=\mathbb{M}_{m',\Phi,\Psi}$, where
$\Lambda_{i}f=\sum_{k \in K_{i}}\langle f,\psi_{i,k}\rangle
e_{i,k}$, $\Psi=\{\psi_{i,k}: i\in J,k \in K_{i}\},
\Phi=\{\phi_{i,k}: i\in J,k \in K_{i}\}$ and $m'=\{m'_{i,k}=m_{i}:
i\in J,k \in K_{i}\}$. Also we have$$\sum_{i \in
J}\|m_{i}\Lambda_{i}f\|^{2}=\sum_{i \in J}\sum_{k \in
K_{i}}|\langle f,m_{i}\psi_{i,k}\rangle|^{2}.$$
 By Proposition 4.3 of
\cite{p.Balazs.09}, $m'\Psi$ satisfies the lower frame condition
for $\mathcal{H}$ with bound $\frac{1}{B_{\Theta}
\|M^{-1}_{m',\Phi,\Psi}\|^{2}}=\frac{1}{B_{\Theta} \|M^{-1}_{m,
\Lambda, \Theta}\|^{2}}$.

(ii) and (iii) follow from (i).

(iv) Let $\Theta$ and $\Lambda$ be g-Bessel sequences for
$\mathcal{H}$ and $m \in
 \ell^{\infty}$. Then $m\Lambda$ and $m\Theta$ are also g-Bessel for
 $\mathcal{H}$.
\end{proof}
In the following proposition we give a sufficient condition for
invertibility  of g-multipliers.
\begin{prop}\label{thm19}
Let $\Lambda=\{ \Lambda_{i} \}_{i \in J}$ be a g-frame for
$\mathcal{H}$, $G : \mathcal{H} \rightarrow \mathcal{H}$ be a
bounded bijective operator and $\Theta_i = \Lambda_iG$, $\forall i
\in J$. Let $m$ be positive (resp. negative) semi-normalized. Then
$\Theta$ is a g-frame for $\mathcal{H}$ and the g-frame multiplier
$M_{m, \Lambda, \Theta}$ is invertible on $\mathcal{H}$ with
$$M^{-1}_{m, \Lambda, \Theta} = \left\{ \begin{array}{ll}
G^{-1}S^{-1}_{(\sqrt{m_i} \Lambda_i)}, & \mbox{when}~~ m_i > 0,
\forall i,\\
 -G^{-1}S^{-1}_{(\sqrt{|m_i|} \Lambda_i)}, & \mbox{when}~~ m_i < 0,
\forall i.
\end{array} \right.  \eqno(1)$$
 \end{prop}
 \begin{proof}  We
have $\Lambda_{i}f=\sum_{k \in K_{i}}\langle f,\psi_{i,k}\rangle
e_{i,k}$, where $\{\psi_{i,k}\}$  is a frame for $\mathcal{H}$.
Since $\Theta_i = \Lambda_iG,$
$$ \Theta_{i}f=\sum_{k \in K_{i}}\langle Gf,\psi_{i,k}\rangle
e_{i,k}=\sum_{k \in K_{i}}\langle f,G^{\star}\psi_{i,k}\rangle
e_{i,k}.$$
We know $G^{\star}$ is a bounded bijective operator, and hence
$\phi_{i,k}=G^{\star}\psi_{i,k}$ is a frame for $\mathcal{H}$
(see\cite{O.Christensen.03}). So $\Theta$ is a g-frame for
$\mathcal{H}$. By Proposition \ref{thm14}
$\mathbb{M}_{m,\Lambda,\Theta}=\mathbb{M}_{m',\Phi,\Psi}$ where
$\Psi=\{\psi_{i,k}: i\in J,k \in K_{i}\}, \Phi=\{\phi_{i,k}: i\in
J,k \in K_{i}\}$ and $m'=\{m'_{i,k}=m_{i}: i\in J,k \in K_{i}\}$.
If $m$ is positive then, by Theorem \ref{thm5}, the g-frame
operator for $\{\sqrt{m_{i}}\Lambda_{i} \}_{i \in J}$ coincides
with the frame operator for $\{\sqrt{m'_{i,k}}\psi_{i,k}: i \in J,
k \in K_{i} \}$. Therefore by Lemma 4.4 of \cite{p.Balazs.06},
$\mathbb{M}_{m,\Lambda,\Theta}=\mathbb{M}_{m',\Phi,\Psi}$ is
invertible and equation (1) holds.
\end{proof}
By the same method that we use to prove Proposition \ref{thm19},
one can prove the followings.
\begin{prop}\label{thm20}
Let $\Lambda=\{ \Lambda_{i} \}_{i \in J}$ be a g-frame for
$\mathcal{H}$ and let $\Lambda^{d} =\{ \Lambda^{d}_{i} \}_{i \in
J}$ be a dual g-frame of $\Lambda$. Let $0 \leq \lambda <
\frac{1}{\sqrt{B_\Lambda B_{\Lambda^d}}} (\leq 1)$ and let $(m_i)$
be such that $1 - \lambda \leq m_i \leq 1 + \lambda$, for all $i
\in J$. Then $M_{m, \Lambda,\Lambda^d}$ and $M_{m, \Lambda^d,
\Lambda}$ are invertible on $\mathcal{H}$,
$$\frac{1}{1 + \lambda \sqrt{B_\Lambda B_{\Lambda^d}}} \|h\| \leq
\|M^{-1}_{m, \Lambda, \Lambda^d} h\| \leq \frac{1}{1 - \lambda
\sqrt{B_\Lambda B_{\Lambda^d}}} \|h\|, \forall h \in \mathcal{H},
\eqno(2)$$ and the same inequalities hold for $\|M^{-1}_{m,
\Lambda^d, \Lambda} h\|$. Moreover,
$$M^{-1}_{m, \Lambda, \Lambda^d} = \sum^{\infty}_{k=0} (M_{m',
\Lambda, \Lambda^d})^{k} \; \mbox{and} \; M^{-1}_{m, \Lambda^d,
\Lambda} = \sum^{\infty}_{k=0} (M_{m', \Lambda^d,
\Lambda})^k,\eqno(3)$$ where $m'=(1-m_i)$.
\end{prop}
\begin{cor}\label{thm21}
Let $\Lambda$ be a g-frame for $\mathcal{H}$ and $\tilde{\Lambda}$
be the canonical dual of $\Lambda$. Let $0 \leq \lambda  <
\sqrt{\frac{A_{\Lambda}}{B_\Lambda}} (\leq 1)$ and $(m_i)_{i \in
J}$ be such that $1 - \lambda \leq m_i \leq 1 + \lambda$, for all
$i \in J$. Then $M_{m, \Lambda, \tilde{\Lambda}}$ and $M_{m,
\tilde{\Lambda}, \Lambda}$ are invertible on $\mathcal{H}$,
$$\frac{1}{1 + \lambda \sqrt{B_\Lambda /A_\Lambda}} \|h\| \leq
\|M^{-1}_{m, \Lambda, \tilde{\Lambda}} h\| \leq \frac{1}{1 -
\lambda \sqrt{B_\Lambda /A_\Lambda}} \|h\|, \forall h \in
\mathcal{H},$$
 and the same inequalities hold for $\|M^{-1}_{m, \tilde{\Lambda}, \Lambda} h\|$.
Moreover,
$$M^{-1}_{m, \Lambda, \tilde{\Lambda}} = \sum^{\infty}_{k=0} (M_{m',
\Lambda, \tilde{\Lambda}})^k \; ~~\mbox{and}\; ~~M^{-1}_{m,
\tilde{\Lambda}, \Lambda} = \sum^{\infty}_{k=0} (M_{m',
\tilde{\Lambda}, \Lambda})^k,$$ where $m'=(1-m_i)$.
\end{cor}
\begin{prop}\label{thm22}
Let $\Lambda$ be a g-frame for $\mathcal{H}$. Assume that $\Theta
- \Lambda$ is a g-Bessel sequence for $\mathcal{H}$ with bound
$B_{\Theta - \Lambda} < \frac{A^2_\Lambda}{B_\Lambda}$. For every
positive (or negative) semi-normalized sequence $m$, satisfying
$$0 < a \leq |m_i| \leq b, \forall i, \; \mbox{and} \; ~~\frac{b}{a}
< \frac{A_\Lambda}{\sqrt{B_{\Theta- \Lambda} B_\Lambda}} ~,$$
 it follows that $\Theta$ is a g-frame for $\mathcal{H}$, the multipliers
 $M_{m, \Lambda, \Theta}$ and $M_{m, \Theta, \Lambda}$ are invertible on
 $\mathcal{H}$,
 $$\frac{1}{bB_\Lambda + b\sqrt{B_\Lambda B_{\Theta - \Lambda}}} \|h\| \leq
 \|M^{-1}_{m, \Lambda, \Theta} h\| \leq \frac{1}{aA_\Lambda -
 b\sqrt{B_\Lambda B_{\Theta-\Lambda}}} \|h\|,$$
 and the same inequalities hold for $\|M^{-1}_{m, \Theta, \Lambda}
 h\|$. Moreover,
 $$M^{-1}_{m, \Lambda, \Theta} = \left\{ \begin{array}{ll}
 \sum^{\infty}_{k=0} [S^{-1}_{(\sqrt{m_i} \Lambda_i)} (S_{(\sqrt{m_i}
 \Lambda_i)} - M_{m, \Lambda, \Theta})]^k S^{-1}_{(\sqrt{m_i} \Lambda_i)}, &
 \mbox{if} ~~\; m_i > 0, \forall i, \\
 -\sum^{\infty}_{k=0} [S^{-1}_{(\sqrt{|m_i|} \Lambda_i)} (S_{(\sqrt{|m_i|}
 \Lambda_i)} + M_{m, \Lambda, \Theta})]^k S^{-1}_{(\sqrt{|m_i|} \Lambda_i)}, &
 \mbox{if}~~ \; m_i < 0, \forall i,
\end{array} \right.$$
$$M^{-1}_{m, \Theta, \Lambda} = \left\{ \begin{array}{ll}
 \sum^{\infty}_{k=0} [S^{-1}_{(\sqrt{m_i} \Lambda_i)} (S_{(\sqrt{m_i}
 \Lambda_i)} - M_{m, \Theta, \Lambda})]^k S^{-1}_{(\sqrt{m_i} \Lambda_i)}, &
 \mbox{if}~~ \; m_i > 0, \forall i, \\
 -\sum^{\infty}_{k=0} [S^{-1}_{(\sqrt{|m_i|} \Lambda_i)} (S_{(\sqrt{|m_i|}
 \Lambda_i)} + M_{m, \Theta, \Lambda})]^k S^{-1}_{(\sqrt{|m_i|} \Lambda_i)}, &
 \mbox{if}~~ \; m_i < 0, \forall i.
\end{array} \right.$$
\end{prop}
 \begin{prop}\label{thm23}
Let $\Lambda$ be a g-frame for $\mathcal{H}$. Assume that \\
$\exists \mu \in [0, \frac{A^2_\Lambda}{B_\Lambda})$ such that
$\sum \| (m_i \Theta_i - \Lambda_i)f\|^2 \leq
\mu\|f\|^2$, ~~$\forall f \in \mathcal{H}$. \\
Then $m \Theta$ is a g-frame for $\mathcal{H}$, the multipliers
$M_{m, \Lambda, \Theta}$ and  $M_{m, \Theta, \Lambda}$ are
invertible on $\mathcal{H}$,
$$\frac{1}{B_\Lambda + \sqrt{\mu B_\Lambda}} \|h\| \leq \|M^{-1}_{m,
\Lambda, \Theta} h\| \leq \frac{1}{A_\Lambda - \sqrt{\mu
B_\Lambda}} \|h\|, \forall h \in \mathcal{H},$$
 and the same inequalities hold for $\|M^{-1}_{m, \Theta, \Lambda}
 h\|$. Moreover,
 $$M^{-1}_{m, \Lambda, \Theta} = \sum^{\infty}_{k=0} [S^{-1}_\Lambda
 (S_\Lambda - M_{m, \Lambda, \Theta})]^k S^{-1}_\Lambda,$$ and $$ M^{-1}_{m, \Theta, \Lambda} = \sum^{\infty}_{k=0}
[S^{-1}_\Lambda (S_\Lambda
 - M_{m, \Theta, \Lambda})]^k S^{-1}_\Lambda.$$
 As a consequence, if $m$ is semi-normalized, then $\Theta$ is also
 a g-frame for $\mathcal{H}$.
 \end{prop}
\begin{prop}\label{thm24}
Let $\Lambda$ be a g-frame for $\mathcal{H}$. Assume that there
exists

$\mu \in [0, \frac{1}{B_\Lambda})$ such that $\sum\|(m_i \Theta_i
- \Lambda^d_i)f\|^2 \leq \mu
\|f\|^2$, for all $f \in \mathcal{H}$, \\
for some dual g-frame $\Lambda^d = (\Lambda^d_i)$ of $\Lambda$.
Then $m\Theta$ is a g-frame for $\mathcal{H}$, the bounded
multipliers $M_{m,, \Lambda, \Theta}$ and $M_{m, \Theta, \Lambda}$
are invertible on $\mathcal{H}$,
$$\frac{1}{1+\sqrt{\mu B_\Lambda}} \|h\| \leq \|M^{-1}_{m, \Lambda, \Theta}
h\| \leq \frac{1}{1-\sqrt{\mu B_\Lambda}} \|h\|, \forall h \in
\mathcal{H},$$ and the same inequalities hold for $\|M^{-1}_{m,
\Theta, \Lambda} h\|$. Moreover,
$$M^{-1}_{m, \Lambda, \Theta} = \sum^{\infty}_{k=0} (I_{\mathcal{H}} -
M_{m, \Lambda, \Theta})^k \;~~ \mbox{and} \;~~~ M^{-1}_{m, \Theta,
\Lambda} = \sum^{\infty}_{k=0} (I_{\mathcal{H}} - M_{m, \Theta,
\Lambda})^k.$$ As a consequence, if $m$ is semi-normalized, then
$\Theta$ is also a g-frame for $\mathcal{H}$ .
\end{prop}

\section{Controlled g-frames}
In  this section we introduce the concept of controlled g-frames
and we show that the sequence induced by each controlled g-frame
is a controlled frame. We also show that controlled g-frames are
equivalent to standard g-frames.
\begin{defn}\label{thm25}  A g-frame controlled by an operator
$C \in GL(\mathcal{H})$ or $C$-controlled g-frame is a family of
vectors $\Lambda=\{ \Lambda_{i} \}_{i \in J}$, such that there
exist two
constants   $  m_{CL} > 0$ and $ M_{CL} < \infty$ satisfying\\

$$ m_{CL} \|f\|^{2} \leq \sum _{i \in J}\langle
\Lambda_{i}C^{\star}f,\Lambda_{i}f \rangle \leq M_{CL}\|f\|^{2},$$
for all $f \in \mathcal{H}$. The controlled g-frame operator is
defined by
$$S_{C}f=\sum_{i \in J}\Lambda_{i}^{\star}\Lambda_{i}C^{\star}f, f \in \mathcal{H}.$$
\end{defn}

\begin{prop}\label{thm7}
The sequence induced by each  controlled g-frame is a controlled
frame.
\end{prop}
\begin{proof}  Let $\Lambda_{i} \in \mathcal{L}(\mathcal{H},\mathcal{H}_{i})$ and
$\Lambda_{i}f=\sum_{k \in K_{i}}\langle f,\psi_{i,k}\rangle
e_{i,k}$, where $\{\psi_{i,k}: i \in J, k \in K_{i} \}$ is the
sequence induced by $\{\Lambda_{i}  : i \in J\}$ with respect to
$\{e_{i,k}:k \in K_{i}\}$. Hence $\Lambda_{i}C^{\star}f=\sum_{k
\in K_{i}}\langle f,C\psi_{i,k}\rangle e_{i,k}$. Then $$\sum _{i
\in J}\langle \Lambda_{i}C^{\star}f,\Lambda_{i}f \rangle=\sum _{i
\in J}\sum_{k \in K_{i}}\langle f,C\psi_{i,k}\rangle\langle
\psi_{i,k},f\rangle,$$and$$S_{C}f=\sum_{i \in
J}\Lambda_{i}^{\star}\Lambda_{i}C^{\star}f=\sum_{i \in J}\sum_{k
\in K_{i}}\langle f,C\psi_{i,k}\rangle \psi_{i,k}.$$

\end{proof}

\begin{prop}\label{thm26} Let $C \in
GL(\mathcal{H})$ and $\Lambda=\{ \Lambda_{i} \}_{i \in J}$ be a
$C$-controlled g-frame in $\mathcal{H}$. Then $\Lambda$ is a
classical g-frame. Furthermore if $S$ is a g-frame operator we
have $SC^{\star} = CS$ and so
$$\sum_{i \in J}\Lambda_{i}^{\star}\Lambda_{i}C^{\star}f=C\sum_{i
\in J}\Lambda_{i}^{\star}\Lambda_{i}f.$$

\end{prop}
\begin{proof}
 Since $\Lambda=\{\Lambda_{i} \}_{i \in J}$ is a $C$-controlled g-frame, by Proposition \ref{thm7}, we
have $\Lambda_{i}f=\sum_{k \in K_{i}}\langle f,\psi_{i,k}\rangle
e_{i,k}$, where $\{\psi_{i,k}\}$ is a $C$-controlled frame. By
Proposition 3.2 of \cite{p.Balazs.06}, $\{\psi_{i,k}\}$ is a
classical frame and so $\Lambda=\{ \Lambda_{i} \}_{i \in J}$ is a
classical g-frame. By Theorem \ref{thm5}, the g-frame operator for
$\{\Lambda_{i}\}_{i \in J}$ coincides with the frame operator for
$\{\psi_{i,k}: i \in J, k \in K_{i} \}$ and the proof is complete
by using Proposition \ref{thm7} and  Proposition 3.2 of
\cite{p.Balazs.06}.
\end{proof}
\begin{prop}\label{thm27}Let $C \in GL(\mathcal{H})$ be self-adjoint. The family $\Lambda$  is a g-frame for
$\mathcal{H}$ controlled by $C$ if and only if it is a (classical)
g-frame for $H$, $C$ is positive and commutes with the g-frame
operator $S$.
\end{prop}
\begin{proof}
The assertion follows from Propositions \ref{thm7}, \ref{thm26}
and Proposition Proposition 3.3 of \cite{p.Balazs.06}.
\end{proof}
\begin{cor}\label{thm28} Let $C$ be a self-adjoint operator and $\Lambda$ be a C-controlled
g-frame. Denote by $(m_{CS},M_{CS})$, $(m,M)$ and $(m_{C},M_{C})$
any bounds for the controlled g-frame operator $S_{C}$, the
g-frame operator $S$, and the operator $C$, respectively. Then,

\begin{itemize}
\item[i)] $m'= \frac{m_{CL}}{M_{C}}$, $ M' =\frac{ M_{CL}}{m_{C}}$
are bounds for $S$;

\item[ii)] $m'_{C} = \frac{m_{CL}}{M}$ , $M'_{C }= \frac{M_{CL}
}{m}$ are bounds for $C$;

\item[iii)] $m'_{CL }= mm_{C}$, $M'_{CL }= MM_{C}$ are bounds for
$S_{C}$.

\end{itemize}
\end{cor}

\section{weighted g-frames}

In  this section we investigate the concept of weighted g-frames,
and show that the sequence induced by each weighted g-frame is a
weighted frame.


\begin{defn}\label{thm29}Let $\{\Lambda_{i}\}_{i \in J}$ and $(w_{i} : i \in J)$ be a
sequence of positive weights. This pair is called a weighted
g-frame or a $w$-g-frame of $\mathcal{H}$ if there exist constants
$m>0$, $M<\infty$ such that
$$m\|f\|^{2}\leq\sum_{i\in J}w_{i}^{2}\|\Lambda_{i}f\| ^{2}\leq M\|
f\|^{2}, f\in \mathcal{H}.$$ Assume now that the restriction on
the weights is lifted, i.e., $(w_{i})\subseteq \mathbb{C}$. Then
we call $(w_{i}\Lambda_{i})$ a weighted g-frame if this sequence
forms a g-frame, i.e.,
$$m\|f\|^{2}\leq\sum_{i\in J}|w_{i}|^{2}\|\Lambda_{i}f\| ^{2}\leq
M\| f\|^{2}, f\in \mathcal{H}.$$

\end{defn}
\begin{prop}\label{thm40} The sequence
induced by each weighted  g-frame is a weighted frame.
\end{prop}

\begin{proof} If $\{w_{i}\Lambda_{i} \}_{i \in J}$ is a weighted
g-frame then, we have $w_{i}\Lambda_{i}f=\sum_{k \in K_{i}}\langle
f,w_{i}\psi_{i,k}\rangle e_{i,k}$ and so
$$\sum_{i\in J}|w_{i}|^{2}\|\Lambda_{i}f\| ^{2}=\sum_{i\in
J}\sum_{k \in K_{i}}|w_{i}|^{2}|\langle f,\psi_{i,k}\rangle|^{2}.
$$ So $\{w'_{i,k}\psi_{i,k}:i \in J , k \in K_{i}\}$ is a weighted
frame with $w'_{i,k}=w_{i}$ for $i \in J , k \in K_{i}$.

\end{proof}

\begin{prop}\label{thm31} Let $C \in GL(\mathcal{H})$ be self-adjoint and
$\{\Lambda_{i}\}_{i \in J}$ be a controlled g-frame and assume
$C\Lambda_{i}^{\star}=w_{i}\Lambda_{i}^{\star}$. Then the sequence
$(w_{i})$ is semi-normalized and positive. Furthermore $C =
M_{w,\Lambda,\widetilde{\Lambda}}.$
\end{prop}

\begin{proof} We have
$\Lambda_{i}f=\sum_{k \in K_{i}}\langle f,\psi_{i,k}\rangle
e_{i,k}$ and so $\Lambda_{i}^{\star}f_{i}=\sum_{k \in
K_{i}}\langle f_{i},e_{i,k}\rangle \psi_{i,k}$ for all $f_{i} \in
\mathcal{H}_{i}$. Since
$C\Lambda_{i}^{\star}=w_{i}\Lambda_{i}^{\star}$, it is easy to
show that $C \psi_{i,k}=w_{i}\psi_{i,k}$. The conclusions follow
from Propositions \ref{thm14}, \ref{thm40} and  Proposition 4.2 of
\cite{p.Balazs.06}.
\end{proof}
The following Lemma can be proved by the same manner.
\begin{lem}\label{thm32} Let  $(w_{i} : i \in J)$ be a semi-normalized  real sequence with bounds $a,b$. Then
if  $\{\Lambda_{i}\}_{i \in J}$  is a g-frame with bounds $m$ and
$M$, then $\{w_{i}\Lambda_{i} \}_{i \in J}$ is also a g-frame with
bounds $a^{2}m$ and $b^{2}M$. The sequence
$\{w_{i}^{-1}\widetilde{\Lambda_{i}} \}_{i \in J}$ is a dual
g-frame of $\{w_{i}\Lambda_{i} \}_{i \in J}$.
\end{lem}

\begin{lem}\label{thm33} Let $\Lambda=\{\Lambda_{i}\}_{i \in J}$ be a g-frame and $w=(w_{i} : i \in J)$ be a positive
semi-normalized sequence. Then the multiplier $M_{w,\Lambda}$ is
the g-frame operator of the g-frame $\{\sqrt{w_{i}}\Lambda_{i}
\}_{i \in J}$ and therefore it is positive, self-adjoint and
invertible.
\end{lem}

\begin{proof} By using Lemma \ref{thm32}, $\{\sqrt{w_{i}}\Lambda_{i} \}_{i \in J}$ is a g-frame
and if $S$ is the g-frame operator of it then
$$Sf=\sum_{i \in
J}(\sqrt{w_{i}}\Lambda_{i})^{\star}\sqrt{w_{i}}\Lambda_{i}f=\sum_{i
\in J}w_{i} \Lambda_{i}^{\star}\Lambda_{i}f$$ and
$$\mathbb{M}_{w,\Lambda}f=\sum_{i \in J}w_{i}
\Lambda_{i}^{\star}\Lambda_{i}f.$$ Therefore
$\mathbb{M}_{w,\Lambda}=S$ is positive, self-adjoint and
invertible.
\end{proof}
We end this section with the following theorem, which can be proved in exactly the same fashion as above.
\begin{thm}\label{thm34} Let $\Lambda=\{\Lambda_{i}\}_{i \in J}$ be a sequence and $w=(w_{i} : i \in J)$ be a positive,
semi-normalized sequence. Then the following properties are
equivalent:

\begin{itemize}
\item[i)] $\Lambda=\{\Lambda_{i}\}_{i \in J}$ is a
g-frame;\\
 \item[ii)] $\mathbb{M}_{w,\Lambda}$ is positive,
self-adjoint and
invertible operator;\\
 \item[iii)] There are constants $m>0$, $M<\infty$ such that\\

$$m\|f\|^{2}\leq\sum_{i\in J}w_{i}\|\Lambda_{i}f\| ^{2}\leq M\|
f\|^{2},$$
for all $f\in \mathcal{H}$;\\
 \item[iv)]  $\{\sqrt{w_{i}}\Lambda_{i} \}_{i \in J}$ is a
g-frame;\\
 \item[v)] $\mathbb{M}_{w',\Lambda}$ is a positive and
invertible operator, for any positive, semi-normalized sequence
$w'=(w'_{i} : i \in J)$;\\
 \item[vi)] $\{w_{i}\Lambda_{i} \in
\mathcal{L}(\mathcal{H},\mathcal{H}_{i}) : i \in J\}$ is a
g-frame.
\end{itemize}
\end{thm}


\end{document}